\documentclass[12pt]{article}

\usepackage[round]{natbib}
\usepackage{amsthm}
\usepackage{amsmath}
\usepackage{amssymb}
\usepackage{textcomp}
\usepackage[a4paper,twoside,truedimen,left=3cm,right=2.5cm,top=4cm,bottom=4cm,foot=2.5cm]{geometry}

\makeatletter
\renewcommand\section{\@startsection%
{section}{1}{0mm}%
{2\baselineskip}%
{\baselineskip}%
{\normalfont \bfseries}}
\makeatother

\newtheorem{theorem}{Theorem}

\begin{document}

\author{Abdulrahman Ali Abdulaziz \\
\small{Assistant Professor} \\
\small{University of Balamand, P. O. Box 100, Tripoli, Lebanon} \\
\small{\tt abdul@balamand.edu.lb}} 
\title{\large \bf Integer Sequences of the Form $\alpha^n \pm \beta^n$} 
\date{}
\maketitle
\thispagestyle{empty}

\abstract{
Suppose that we want to find all integer sequences of the form $\alpha^n +
\beta^n$, where $\alpha$ and $\beta$ are complex numbers and $n$ is a
nonnegative integer. Since $\alpha^0 + \beta^0$ is always an integer, our
task is then equivalent to determining all complex pairs $(\alpha,\beta)$
such that
\begin{equation}
\alpha^n + \beta^n \in \mathbb Z, \quad n > 0. \label{Main:Cond}
\end{equation}
Let $p$ and $q$ be two integers; and consider the quadratic equation
\begin{equation}
z^2 - pz -q = 0. \label{Main:Eqn}
\end{equation}
By the quadratic formula, the roots of (\ref{Main:Eqn}) are
\begin{equation}
r = \frac{p}{2} + \frac{\sqrt{p^2 + 4q}}{2} \qquad \text{and} \qquad s =
\frac{p}{2} - \frac{\sqrt{p^2 + 4q}}{2}. \label{Main:Roots}
\end{equation}
In this paper, we prove that there is a correspondence between the roots of
(\ref{Main:Eqn}) and integer sequences of the form $\alpha^n + \beta^n$. In
addition, we will show that there are no integer sequences of the form
$\alpha^n - \beta^n$. Finally, we use special values of $\alpha$ and $\beta$
to obtain a range of formulas involving Lucas and Fibonacci numbers.
}

\section{Sums of Like Powers}
In this section, we prove that the complex pair $(\alpha,\beta)$ satisfies
(\ref{Main:Cond}) if and only if $\alpha$ and $\beta$ are the roots of
(\ref{Main:Eqn}).

\begin{theorem} \label{Main:Thm}
  If $r$ and $s$ are the roots of (\ref{Main:Eqn}), then $r^n + s^n$ is an
  integer for every natural number $n$.
\end{theorem}

\begin{proof}
  By the binomial theorem, we have
\begin{align*}
  r^n + s^n& = \left(\frac{p}{2} + \frac{\sqrt{p^2 + 4q}}{2} \right)^n +
  \left(\frac{p}{2} - \frac{\sqrt{p^2 + 4q}}{2} \right)^n \\
  & = \frac{1}{2^n} \sum_{i=0}^n \binom{n}{i} p^{n-i} \left(
    \left(\sqrt{p^2+4q}\right)^i + \left(-\sqrt{p^2+4q}\right)^i \right).
\end{align*}
Let $\Delta =\sqrt{p^2+4q}$. Then $\Delta^i + (-\Delta)^i = 0$ for odd values
of $i$ and $\Delta^i + (-\Delta)^i = 2\Delta^i$ for even values of $i$.
Hence, replacing $i$ by $2i$ in the last summation, the upper limit for $i$
becomes $\lfloor n/2 \rfloor$. This yields
\[
  r^n+s^n  = \frac{1}{2^n} \sum_{i=0}^{\lfloor n/2 \rfloor} \binom{n}{2i}
  p^{n-2i} \left(2\Delta^{2i}\right),
\]
which is equivalent to 
\begin{equation}
r^n+s^n = \frac{1}{2^{n-1}} \sum_{i=0}^{\lfloor n/2 \rfloor} \binom{n}{2i}
  p^{n-2i} \, (p^2+4q)^i. \label{Exp1}
\end{equation}
Now expanding $(p^2+4q)^i$ in the last equation, we obtain
\begin{align*}
  r^n+s^n& = \frac{1}{2^{n-1}} \sum_{i=0}^{\lfloor n/2 \rfloor} \binom{n}{2i}
  p^{n-2i} \sum_{k=0}^i \binom{i}{k}(p^2)^{i-k}(4q)^k
  \\
  &= \frac{1}{2^{n-1}} \sum_{i=0}^{\lfloor n/2 \rfloor} \binom{n}{2i}
  p^{n-2i} \sum_{k=0}^i \binom{i}{k} p^{2i-2k} 2^{2k} q^k 
\\
&= \frac{1}{2^{n-1}} \sum_{i=0}^{\lfloor n/2 \rfloor} \binom{n}{2i}
  \sum_{k=0}^i \binom{i}{k} 2^{2k} p^{n-2k}q^k.
\end{align*}
Since $\binom{n}{2i}$ is multiplied by $\binom{i}{k}$ for $k \leq i \leq
\lfloor n/2 \rfloor$, the last summation can be rearranged so that the
coefficient of $p^{n-2k}q^k$ is
\begin{equation}
\frac{1}{2^{n-2k-1}} \, \sum_{i=k}^{\lfloor n/2 \rfloor} \binom{n}{2i}
\binom{i}{k},  \label{Coeff}
\end{equation}
where $0 \leq k \leq \lfloor n/2 \rfloor$. It follows that
\[
r^n+s^n = \sum_{k=0}^{\lfloor n/2 \rfloor} \frac{1}{2^{n-2k-1}}
\sum_{i=k}^{\lfloor n/2 \rfloor} \binom{n}{2i} \binom{i}{k} p^{n-2k}q^k. 
\]
Since $p$ and $q$ are integers, the proof would be completed provided one can
show that (\ref{Coeff}) yields only integer values. We do so by showing that 
\begin{equation}
2^{2k-n+1} \, \sum_{i=k}^{\lfloor n/2 \rfloor} \binom{n}{2i}
\binom{i}{k} = \frac{n}{n-k} \binom{n-k}{k}, \label{IdenL}
\end{equation} 
where it is clear that the right-hand side of (\ref{IdenL}) is always an
integer. Observe that once (\ref{IdenL}) is proved, we can write
\begin{equation}
r^n+s^n = \sum_{k=0}^{\lfloor n/2 \rfloor} \frac{n}{n-k} \binom{n-k}{k}
p^{n-2k}q^k. \label{Exp2}
\end{equation}

An easy way to prove (\ref{IdenL}) is through the help of a computer algebra
system. For example, the answer to the left-hand side of (\ref{IdenL}) 
given by \emph{Mathematica~5.2} is 
\[
\frac{n\Gamma(n-k)}{\Gamma(k+1)\Gamma(n-2k+1)},
\]
where $\Gamma$ is the well known generalized factorial function. Using the
fact that $\Gamma(n) = (n-1)!$ for the positive integer $n$, the answer can
be written as
\[
\frac{n(n-k-1)!}{k!(n-2k)!} = \frac{n}{n-k} \left(\frac{(n-k)!}{k!(n-2k)!}
\right) = \frac{n}{n-k} \binom{n-k}{k},
\]
and so the proof is complete\footnote{\citet*{Draim} proved Theorem
  \ref{Main:Thm} using mathematical induction. However, the direct proof
  given here has the advantage that its steps, as we shall see, can be used
  to calculate some interesting binomial sums.}.
\end{proof}

Equation (\ref{IdenL}) can be proved and thus written in many other ways. For
example, Draim and Bickell \citet*{Draim} proved that
\begin{equation}
r^n + s^n = \sum_{k=0}^{\lfloor n/2 \rfloor} \left( 2 \binom{n-k}{k} -
\binom{n-k-1}{k} \right) p^{n-2k}q^k. \label{Exp3}
\end{equation}
Then using properties of binomial coefficients \citet*{Koshy} showed that the
coefficient of $p^{n-2k}q^k$ in (\ref{Exp3}) is
\[
\binom{n-k}{k} + \binom{n-k-1}{k-1} = \frac{n}{n-k} \binom{n-k}{k}. 
\]
Similar results are proved in \citep*{Benoum,Woko}.  More generally,
\citet*{Hirsh} used summable hypergeometric series to prove, among 
other identities, that
\[
\sum_{i=k}^{\lfloor n/2 \rfloor} \binom{n}{2i} \binom{i}{k} = 2^{n-2k-1}
\left( \binom{n-k}{k} + \binom{n-k-1}{k-1} \right).
\]
In fact, (\ref{IdenL}) is one of a whole class of identities involving
hypergeometric series that can be proved by means of well established
algorithms; see \citet*{Petkov} for a survey of such algorithms. Lastly, if $x
= \sqrt{-p^2/q}$, then it can be shown that
\begin{equation}
r^n + s^n = 2p^n x^{-n} T_n(x/2), \label{Cheby1}
\end{equation}
where $T_n(x)$ is Chebyshev polynomial of the first kind of order $n$ defined
by the recurrence relation 
\[
T_0(x) = 1, \quad T_1(x) = x, \quad \text{and} \quad T_{n+1}(x) = 2xT_n(x) -
T_{n-1}(x). 
\]

The converse of Theorem \ref{Main:Thm} is also true. That is, if $\alpha^n +
\beta^n$ is an integer for every natural number $n$, then $\alpha$ and
$\beta$ are the zeros of a quadratic polynomial with integer coefficients.
Since $\alpha$ and $\beta$ are the roots of the equation
\[
z^2 -(\alpha+\beta)z + \alpha\beta = 0
\] 
and since $\alpha + \beta$ is an integer, the proof would be completed
provided one can show that $\alpha\beta$ is also an integer. But we know that
$2\alpha \beta$ is an integer since
\[
2\alpha \beta = (\alpha + \beta)^2 - (\alpha^2 + \beta^2) \in \mathbb Z. 
\]
It follows that if $\alpha\beta$ is not an integer, then $\alpha\beta = m/2$
for some odd integer $m$. Now the fact that $\alpha^4 + \beta^4$ is an
integer implies that
\[
2\alpha \beta(3\alpha \beta + 2\alpha^2 + 2\beta^2) = (\alpha +
\beta)^4 - (\alpha^4 + \beta^4) \in \mathbb Z.
\] 
But this could not hold unless $6\alpha^2 \beta^2 = 3m^2/2$ is an integer,
which is impossible when $m$ is odd. We conclude that $\alpha \beta$ must be
an integer, as required.

Finally, observe that if $m$ is an nonzero integer, then $\alpha = m+r$ and
$\beta = m+s$ are the roots of the equation
\[
z^2 - (p+2m)z - (q-pm-m^2) = 0. 
\]
Therefore, $\alpha^n + \beta^n$ is an integer for every positive integer $n$.
By the binomial theorem, we have
\begin{equation}
\alpha^n + \beta^n = (m+r)^n + (m+s)^n = \sum_{i=0}^n \binom{n}{i} m^{n-i}
(r^i + s^i) \in \mathbb Z. \label{Main:Sum} 
\end{equation}

\section{Differences of Like Powers}
Let $\alpha = x+iy$ and $\beta = u+iv$ be two complex numbers. Apart from the
trivial case $\alpha = \beta$ and the case when both $\alpha$ and $\beta$ are
integers, we will show that there are no integer sequences of the form
$\alpha^n - \beta^n$. To see this, observe that if $\alpha - \beta$ is an
integer then $v=y$, and so
\[
\alpha^2 - \beta^2 = (x^2-u^2) + 2y(x-u)i.
\]
It follows that $2y(x-u)=0$, i.e., either $y=0$ or $u=x$. But for $y=0$,
we get real $\alpha$ and $\beta$; and for $u=x$, we get the trivial solution
$\alpha = \beta$.

Now suppose that $\alpha$ and $\beta$ are real numbers such that $\alpha^n -
\beta^n$ is always an integer. Assuming that $\alpha$ and $\beta$ are not
both integers, then the fact that $\alpha - \beta$ is an integer implies that
there exists a real number $a$ and integers $x$ and $y$ such that $\alpha =
x+a$ and $\beta = y+a$. It follows that
\[
\alpha^2 - \beta^2 = x^2 - y^2 + 2a(x-y) \in \mathbb Z 
\]
only if $a$ is a rational number. Clearly, $a$ should not be an integer since
otherwise $\alpha$ and $\beta$ will be both integers. But if $a$ is a
noninteger rational number, then the same is true for $\alpha = x+a$ and
$\beta=y+a$.
\begin{theorem}
  If $\alpha$ and $\beta$ are two distinct (noninteger) rational numbers,
  then $\alpha^n - \beta^n$ cannot be an integer for every natural number
  $n$. 
\end{theorem} 

\begin{proof}
  The assumption that $\alpha$ and $\beta$ are two distinct rational numbers
  such that $\alpha^n - \beta^n \in \mathbb Z$ is equivalent to saying that
  there exists distinct integers $x$, $y$, and $z$ such that $|z| \neq 1$,
  $\gcd(x,y,z) = 1$, and
\[
\frac{x^n - y^n}{z^n} \in \mathbb Z. 
\]
Since $\gcd(x,y,z) = 1$, we can divide by any common divisor of $x$ and $y$
until we reach $\gcd(x,y) = 1$. Now for $n=1$, we get $z| (x-y)$. Let $p$ be
an arbitrary prime greater than the largest prime divisor of $z$, and define
\[
P_p(x,y) = \frac{x^p-y^p}{x-y} = \sum_{i=0}^{p-1}x^{p-i-1}y^i.
\]
Then $\gcd(z^p,p) = 1$ and $P_p(x,y)$ yields only integers. Since $\gcd(x,y) =
1$ and $p$ is prime, using elementary number theory it can be easily shown
that $\gcd(x-y,P_p(x,y))$ is either $1$ or $p$. It follows that
$\gcd(z,P_p(x,y)) = 1$. This coupled with the fact that $x^p-y^p =
(x-y)P_p(x,y)$ implies that $z^p | (x^p-y^p)$ if and only if $z^p|(x-y)$. But
this forces $x$ to be equal to $y$, a contradiction of the assumption that
$x$ and $y$ are distinct\footnote{It turned out that if $x=z^m+1$ and $y=1$,
  then $(x^n-y^n)/z^n$ is an integer for $n \leq m$, where is $m$ is any
  positive integer. If $m$ is prime then this is the smallest value of $x$
  that yields a solution for $n \leq m$, assuming that $x$ and $y$ are both
  positive. It is only when $m$ is allowed to go to infinity that a solution
  cannot be found.  In fact, it can be shown that one cannot find a positive
  integer $N$ and rational numbers $\alpha$ and $\beta$ such that $\alpha^n -
  \beta^n \in \mathbb Z$ for $n > N$, no matter how large $N$ is.}.
\end{proof}

Having shown that there are no nontrivial integer sequences of the form
$\alpha^n - \beta^n$, it should be mentioned that the same is not true for
sequences of the form 
\[
\frac{\alpha^n-\beta^n}{\alpha-\beta}.
\]
In particular, if we choose $\alpha = r$ and $\beta = s$, then a
Fibonacci-like sequence is generated. More generally, we have the following
result.
\begin{theorem} \label{Thm:Diff}
  If $\alpha = y+r$ and $\beta = y+s$, then $(\alpha^m -
  \beta^m)/(\alpha-\beta)$ is an integer for every positive integer $m$.
\end{theorem}

\begin{proof}
  Since $\alpha - \beta = r-s$, the binomial theorem yields
\[
\frac{\alpha^m - \beta^m}{\alpha-\beta} = \sum_{n=0}^m
\binom{m}{n} y^{m-n} \, \frac{r^n - s^n}{r-s}, 
\]
which is clearly an integer if $(r^n - s^n)/(r-s)$ is an integer. Since
$r^0-s^0 =0$, we need only consider positive powers of $r$ and $s$. Now
following an argument similar to that used in Theorem~\ref{Main:Thm}, we get,
for $n \geq 0$,
\begin{align*}
  \frac{r^{n+1}-s^{n+1}}{r-s}& = \frac{1}{2^{n}} \sum_{i=0}^{\lfloor
    n/2 \rfloor} \binom{n+1}{2i+1} p^{n-2i} \, (p^2+4q)^i  \\
  & =\sum_{k=0}^{\lfloor n/2 \rfloor} \frac{1}{2^{n-2k}}
    \sum_{i=k}^{\lfloor n/2 \rfloor} \binom{n+1}{2i+1} \binom{i}{k}
    p^{n-2k}q^k.
\end{align*}
Again, we can use \emph{Mathematica} to simplify the coefficient of
$p^{n-2k}q^k$ in the last equation. This yields
\[
2^{2k-n} \, \sum_{i=k}^{\lfloor n/2 \rfloor} \binom{n+1}{2i+1}
\binom{i}{k} = \frac{\Gamma(n-k+1)}{\Gamma(k+1)\Gamma(n-2k+1)}. 
\]
Using the properties of $\Gamma$, the right-hand side of the last equation
can be written as 
\[
\frac{(n-k)!}{k! (n-2k)!} = \binom{n-k}{k}.  
\]
We conclude that
\begin{equation}
\frac{r^{n+1}-s^{n+1}}{r-s} = \sum_{k=0}^{\lfloor n/2 \rfloor} \binom{n-k}{k} 
p^{n-2k}q^k. \label{IdenF}
\end{equation}
It is clear that the right-hand side of (\ref{IdenF}) is always an integer,
and so the proof is complete.
\end{proof}

As before, if $(\alpha^n - \beta^n)/(\alpha-\beta)$ is an integer for every
$n$, then $\alpha$ and $\beta$ are the roots of a quadratic equation with
integer coefficients. This is so since for $n=2$, we get $\alpha + \beta \in
\mathbb Z$; while for $n=3$, we get $(\alpha + \beta)^2 - \alpha \beta \in
\mathbb Z$, which implies that $\alpha \beta$ is an integer.  Also, we can
express equation (\ref{IdenF}) in terms of Chebyshev polynomials. In this
case, we get
\[
\frac{r^{n+1}-s^{n+1}}{r-s} = p^n x^{-n} U_n(x/2), 
\]
where $U_n(x)$ is Chebyshev polynomial of the second kind of order $n$
defined by
\[
U_0(x) = 1, \quad U_1(x) = 2x, \quad \text{and} \quad U_{n+1}(x) = 2xU_n(x) -
U_{n-1}(x).
\]

\section{Special Cases}
We have proved that $r^n+s^n$ is an integer for every natural number $n$ if
and only if $r$ and $s$ are the roots of $z^2-pz-q=0$. Depending on the
values of $p$ and $q$, some of the resulting sequences are more interesting
than others. For instance, if $p=0$ then $r = \sqrt{q}$ and $s = -\sqrt{q}$.
In this case, $r^n + s^n$ is either zero (when $n$ is odd) or $2q^{n/2}$
(when $n$ is even). On the other hand, if $q=0$ then $r=p$ and $s=0$, and
thus $r^n+s^n = p^n$. So, suppose that both $p$ and $q$ are different from
zero. Then using the identity $T_n(\theta) = \cos (n \arccos \theta)$ one can
rewrite (\ref{Cheby1}) as
\begin{equation}
r^n + s^n = 2p^n x^{-n} \cos \left(n \arccos \frac{x}{2}
\right). \label{Cheby3} 
\end{equation}
Since $x = \sqrt{-p^2/q}$, we see that the \emph{simplest} sequences are
obtained when both $p$ and $q$ are equal to one in absolute value. 

First, we take $p=q=1$. This yields $r = \phi$ and $s =-\varphi$, where $\phi
= (1+\sqrt5)/2$ is the \emph{golden ratio} and $\varphi = 1/\phi$. Using
mathematical induction, one can easily show that $r^n + s^n = L_n$. This is
the well known \emph{Binet formula} for $n$-th Lucas number $L_n$. In fact,
the formula is a special case of a more general formula that, given $G_0$ and
$G_1$, calculates the $n$-th \emph{generalized} Fibonacci number defined for
$n \geq 2$ by
\[
G_n = pG_{n-1} + qG_{n-2},
\]
where $p$ and $q$ are arbitrary numbers (integers in our case). It turned out
that if $r$ and $s$ are distinct roots of $z^2-pz-q=0$, then
\[
G_n = \frac{(G_1-sG_0)r^n - (G_1-rG_0)s^n}{r-s}, 
\]
see \citet*{Niven}. In particular, if $p=q=1$ then it easily seen that for
$G_0=2$ and $G_1 =1$ we get
\[
G_n = \phi^n + (-\varphi)^n = L_n,
\]
while for $G_0=0$ and $G_1 =1$ we get 
\[
G_n = \frac{\phi^n - (-\varphi)^n}{\sqrt5} = F_n,
\]
where $F_n$ is the $n$-th Fibonacci number. More generally, if $G_0=0$ and
$G_1 =1$, then
\[
G_n = \frac{r^n - s^n}{r-s} 
\]
for any $p$ and $q$. On the other hand, if $G_0=2$ and $G_1 =1$, then
\[
G_n = r^n + s^n
\]
for any $q$, provided that $p=1$.  

Beside the identity $\phi^n + (-\varphi)^n = L_n$, we can use the steps of
Theorem~\ref{Main:Thm} to develop other formulas for $L_n$. For instance,
setting $p = q = 1$ in (\ref{Exp1}) yields
\[
L_n = \frac{1}{2^{n-1}} \sum_{i=0}^{\lfloor n/2 \rfloor} \binom{n}{2i} 5^i. 
\]
Doing the same in (\ref{Exp2}) we get
\[
L_n = \sum_{k=0}^{\lfloor n/2 \rfloor} \frac{n}{n-k}
\binom{n-k}{k}. 
\]

Next, we take $p=q=-1$. Then the zeros of the corresponding polynomial are 
\[
r = -\frac{1}{2} + \frac{i\sqrt{3}}{2} = \Lambda \qquad \text{and} \qquad
s = -\frac{1}{2} - \frac{i\sqrt{3}}{2} = \lambda. 
\]
Substituting $p=q=-1$ in (\ref{Cheby3}), we get $r^n+s^n = 2(-1)^n
\cos(n\pi/3)$. Starting with $n=1$, it is obvious that $r^n+s^n$ takes on the
cycle $\{-1, -1, 2\}$. More generally, if we let $q = -p^2$, then we obtain
\[
x = 1, \quad r = \Lambda p \quad \text{and} \quad  s = \lambda p, 
\]
and so $r^n + s^n = 2p^n \cos(n\pi/3)$. On the other hand, $q = p^2$ gives
\[
x = i, \quad r = \phi p \quad \text{and} \quad s = -\varphi p, 
\]
and so $r^n + s^n = L_np^n$. 

So far, we have taken $\alpha = r$ and $\beta = s$, which is equivalent to
setting $m=0$ in (\ref{Main:Sum}). But a whole new set of identities can
obtained by allowing $m$ to be different form zero. Suppose that we fix
$p=q=1$.  Then for $m=1$ we get $\alpha = 1+\phi = \phi^2$ and $\beta =
1-\varphi = \varphi^2$.  Hence, we have
\[
\alpha^n + \beta^n = (\phi^2)^n + (\varphi^2)^n = \phi^{2n} + \varphi^{2n} =
L_{2n}.  
\]
Now setting $r=\phi$, $s=-\varphi$ and $m=1$ in the right-hand side of
(\ref{Main:Sum}) we get
\[
\sum_{i=0}^n \binom{n}{i} \left( \phi^i + (-\varphi)^i \right) = \sum_{i=0}^n
\binom{n}{i} L_i = L_{2n}.
\]
Moreover, since $1+\phi$ and $1-\varphi$ are the zeros of $z^2-3z+1$,
substituting $p=3$ and $q=-1$ in (\ref{Exp1}) we obtain
\[
L_{2n} = \frac{3^n}{2^{n-1}} \sum_{i=0}^{\lfloor n/2 \rfloor}
\left(\frac{5}{9}\right)^i \binom{n}{2i}.
\]
Doing the same in (\ref{Exp2}) yields 
\[
L_{2n} = \sum_{k=0}^{\lfloor n/2 \rfloor} (-1)^k \, 3^{n-2k} \, \frac{n}{n-k}
\binom{n-k}{k}.  
\]
Similarly, for $m=-1$, we get $\alpha = -1+\phi = \varphi$ and $\beta =
-1-\varphi = -\phi$. It follows that
\[
\alpha^n + \beta^n = (-1)^n (\phi^n + (-\varphi)^n) = (-1)^n L_n.
\]
But letting $m=-1$ in (\ref{Main:Sum}) gives 
\[
\varphi^n + (-\phi)^n = (-1)^n L_n = \sum_{i=0}^n \binom{n}{i} (-1)^{n-i} 
L_i. 
\]
Multiplying both sides of the last equation by $(-1)^n$, we deduce that 
\[
\sum_{i=0}^n (-1)^i \binom{n}{i} L_i = L_n. 
\]
Since $\binom{n}{n} = 1$, we obtain 
\[
\sum_{i=0}^{n-1} (-1)^i \binom{n}{i} L_i =
\begin{cases}
0,& \text{if $n$ is even}; \\
2L_n,& \text{if $n$ is odd}. 
\end{cases}
\]

Next, we look at $m=\pm2$. For $m=2$, we get $\alpha = 1+\phi^2 =
\sqrt{5}\phi$ and $\beta = 1+\varphi^2= \sqrt{5}\varphi$. It follows that
\[
\alpha^n + \beta^n = (\sqrt{5}\phi)^n + (\sqrt{5}\varphi)^n = 
\begin{cases} 
  5^kL_{n},& \text{if $n = 2k$};  \\
  5^{k+1}F_{n},& \text{if $n = 2k+1$}. 
\end{cases}
\]
This is so since 
\[
L_{2k} = \phi^{2k} + \varphi^{2k} \quad \text{and} \quad F_{2k+1} =
\frac{\phi^{2k+1} + \varphi^{2k+1}}{\sqrt{5}}.  
\]
Alternatively, setting $\alpha = 1+\phi^2$ and $\beta = 1+\varphi^2$ in
(\ref{Main:Sum}) we obtain
\[
\alpha^n + \beta^n = \sum_{i=0}^{n} \binom{n}{i} \left(\phi^{2i} +
\varphi^{2i}\right) =  \sum_{i=0}^{n} \binom{n}{i} L_{2i}. 
\]
This leads to the known identity 
\[
\sum_{i=0}^{n} \binom{n}{i} L_{2i} 
= 
\begin{cases} 
  5^kL_{n}& \text{if } n = 2k  \\
  5^{k+1}F_{n}& \text{if } n = 2k+1,
\end{cases}
\]
which is proved in \citet*{Vajda}. Since $\alpha = 2+\phi$ and $\beta
=2-\varphi$ are the zeros of $z^2-5z+5$, using (\ref{Exp1}) and (\ref{Exp2})
we respectively get
\[
\sum_{i=0}^{n} \binom{n}{i} L_{2i} = \frac{5^n}{2^{n-1}} \sum_{i=0}^{\lfloor
n/2 \rfloor} \frac{1}{5^i} \binom{n}{2i} = \sum_{k=0}^{\lfloor n/2 \rfloor}
(-1)^k \, 5^{n-k} \, \frac{n}{n-k} \binom{n-k}{k}.  
\]
As for $m=-2$, we obtain $\alpha = -2+\phi = -\varphi^2$ and $\beta = -2-\varphi
= -\phi^2$. Therefore, 
\[
\alpha^n + \beta^n = (-\varphi^2)^n + (-\phi^2)^n = (-1)^n (\varphi^{2n} +
\varphi^{2n}) = (-1)^nL_{2n}.  
\]
Using (\ref{Main:Sum}) we get  
\[
\alpha^n + \beta^n = \sum_{i=0}^{n} (-2)^{n-i} \binom{n}{i} L_i =
(-1)^nL_{2n}.   
\]

Continuing in this way, one can obtain a myriad of formulas involving Lucas
numbers. Moreover, using Theorem \ref{Thm:Diff}, similar results involving
Fibonacci numbers can be obtained as well. In addition, by expressing the
roots of (\ref{Main:Eqn}) in different forms, new sets of identities will
emerge. For example, if $\alpha$ and $\beta$ are the roots of $z^2-z-1=0$,
then they can be written as $\alpha = m+r$ and $\beta = m+s$, where $m=2$ and
$r$ and $s$ are the roots of $z^2+3z+1=0$. Since the roots of first equation
are $\phi$ and $-\varphi$ and those of the second equation are $-\varphi^2$
and $-\phi^2$, we get $(2-\varphi^2)^n + (2-\phi^2)^n = \phi^n + (-\varphi)^n
= L_n$. But substituting for $\alpha$ and $\beta$ in (\ref{Main:Sum}) we
obtain
\[
(2-\varphi^2)^n + (2-\phi^2)^n = \sum_{i=0}^n (-1)^i 2^{n-i}
\binom{n}{i} L_{2i} = L_n. 
\]
Other types of interesting identities can also be deduced. For instance,
setting $k=0$ in (\ref{IdenL}) yields
\begin{equation}
2^{1-n} \sum_{i=0}^{\lfloor n/2 \rfloor} \binom{n}{2i} = 1 \qquad \text{or}
\qquad \sum_{i=0}^{\lfloor n/2 \rfloor} \binom{n}{2i} =
2^{n-1}. \label{Spec1} 
\end{equation}
Similarly, when $k=1$ and $n > 1$, we get
\begin{equation}
2^{3-n} \sum_{i=1}^{\lfloor n/2 \rfloor} i\binom{n}{2i} = n \qquad \text{or}
\qquad \frac{4}{n} \sum_{i=0}^{\lfloor n/2 \rfloor} i\binom{n}{2i} =
2^{n-1}. \label{Spec2}
\end{equation}
Equating (\ref{Spec1}) with (\ref{Spec2}) we obtain
\[
n\sum_{i=0}^{\lfloor n/2 \rfloor} \binom{n}{2i} = 4
\sum_{i=0}^{\lfloor n/2 \rfloor} i\binom{n}{2i}. 
\]

\bibliographystyle{plainnat}

\begin{thebibliography}{6} 

\bibitem[Benoumhani(2003)]{Benoum} Benoumhani, M. (2003). \textquotesingle A
  sequence of binomial coefficients related to Lucas and Fibonacci
  numbers.\textquotesingle \ \textit{Journal of Integer Sequences}
  6(2). Available at http://www.cs.uwaterloo.ca/journals/JIS/VOL6/Benoumhani 
  [5 June 2003]. 

\bibitem[Draim \& Bickell(1996)]{Draim}Draim, N.A. \& M. Bickell
  (1966). \textquotesingle Sums of $n$-th powers of roots of a 
  given quadratic equation.\textquotesingle \ \textit{Fibonacci Quarterly} 4:
  170--178.

\bibitem[Hirschhorn(2002)]{Hirsh}Hirschhorn, M.D. (2002). \textquotesingle
  Binomial coefficient identities and hypergeometric series.\textquotesingle
  \ \textit{Australian Mathematical Society Gazette} 
  29: 203--208. 

\bibitem[Koshy(2001)]{Koshy}Koshy, T. (2001). \textit{Fibonacci and Lucas
    Numbers with Applications}. New York: John Wiley \& Sons.
  
\bibitem[Niven \& Zuckerman(1980)]{Niven}Niven, I. \& H. S. Zuckerman
  (1980). \textit{An Introduction to the Theory of Numbers}. 4th ed. New York:
  John Wiley \& Sons. 
  
\bibitem[Petkov{\u{s}}ek(1996)]{Petkov}Petkov{\u{s}}ek, M., H.S. Wilf \&
  D. Zeilberger (1996). \textit{A = B}. Wellesley: A K Peters.
  
\bibitem[Vajda(1989)]{Vajda}Vajda, S. (1989). {\itshape Fibonacci and Lucas
    Numbers, and  the Golden Section}. Chichester: Ellis Horwood.

\bibitem[Woko(1997)]{Woko}Woko, E.J. (1997). \textquotesingle A Pascal-like triangle for $\alpha^n +
  \beta^n$.\textquotesingle \ \textit{Mathematical Gazette} 81: 75--79.
  
\end{thebibliography}

\end{document}